\documentclass{amsart}
\usepackage{amsfonts}
\usepackage{amsmath}

\newtheorem{theorem}{Theorem}
\theoremstyle{plain}

\newtheorem{corollary}{Corollary}

\newtheorem{example}{Example}

\newtheorem{lemma}{Lemma}

\newtheorem{proposition}{Proposition}
\newtheorem{remark}{Remark}

\numberwithin{equation}{section}
\input{tcilatex}

\begin{document}
\title[Simpson's inequalities]{An Improvement inequality of Simpson's type for Quasi-Convex Mappings with Applications}
\author[M.W. Alomari]{Mohammad W. Alomari}
\address{Department of Mathematics,
Faculty of Science and Information Technology, Irbid National
University, 2600 Irbid 21110, Jordan} \email{mwomath@gmail.com}

\date{September 5, 2010.}

\begin{abstract}
In this paper, an inequality of Simpson type for quasi-convex
mappings are proved. The constant in the classical Simpson's
inequality is improved. Furthermore, the obtained bounds can be
(much) better than some recently obtained bounds. Application to
Simpson's quadrature rule is also given.
\end{abstract}

\maketitle

\section{Introduction}

Suppose $  f:[a,b] \to  \mathbb{R} $, is fourth times continuously
differentiable mapping on $\left( {a,b} \right)$ and $ \left\|
{f^{\left( 4 \right)} } \right\|_\infty  : = \mathop {\sup
}\limits_{x \in \left( {a,b} \right)} \left| {f^{\left( 4 \right)}
\left( x \right)} \right| < \infty$  . The following inequality
\begin{align}
\label{eq1.1}\left|{\int\limits_a^b {f\left( x \right)dx}
-\frac{(b-a)}{6}\left[ {f\left( a \right) + 4 f\left( {\frac{{a +
b}}{2}} \right) + f\left( b \right)} \right]} \right| \le
\frac{\left( {b - a} \right)^5}{{2880}}\left\| {f^{\left( 4
\right)} } \right\|_\infty
\end{align}
holds, and it is well known in the literature as Simpson's
inequality. It is well known that if the mapping $f$ is neither
four times differentiable nor is the fourth derivative $f^{(4)}$
bounded on $(a, b)$, then we cannot apply the classical Simpson
quadrature formula.\\

In \cite{12b}, Pe\v{c}ari\'{c} and Varo\v{s}anec, obtained some
inequalities of Simpson's type for functions whose $n$-th
derivative, $n \in \left\{ {0,1,2,3} \right\}$ is of bounded
variation, as follow:
\begin{theorem}
\label{thm1.1} Let $n \in \left\{ {0,1,2,3} \right\}$. Let $f$ be
a real function on $[a,b]$ such that $f^{(n)}$ is function of
bounded variation. Then
\begin{multline}
\label{eq1.2}\left|{\int\limits_a^b {f\left( x \right)dx}
-\frac{(b-a)}{6}\left[ {f\left( a \right) + 4 f\left( {\frac{{a +
b}}{2}} \right) + f\left( b \right)} \right]} \right| \\ \le C_n
\left( {b - a} \right)^{n + 1} \bigvee_a^b \left( {f^{\left( n
\right)} } \right),
\end{multline}
where, $$C_0  = \frac{1}{3},\,\, C_1  = \frac{1}{{24}},\,\, C_2  =
\frac{1}{{324}}, \,\,C_3  = \frac{1}{{1152}},$$ and $V_a^b \left(
f^{(n)} \right)$ is the total variation of $f^{(n)}$ on the
interval $[a,b]$.
\end{theorem}
Here to note that, the inequality (\ref{eq1.2}) with $n=0$, was
proved in literature by Dragomir \cite{5b}. Also, Ghizzetti and
Ossicini \cite{11b}, proved that if $f'''$ is an absolutely
continuous mapping with total variation $\bigvee_a^b \left( f
\right)$,
then (\ref{eq1.2}) with $n=3$ holds.\\

In recent years many authors were established several
generalizations of the Simpson's inequality for mappings of
bounded variation and for Lipschitzian, monotonic, and absolutely
continuous mappings via kernels, for refinements, counterparts,
generalizations and several Simpson's type inequalities see
\cite{4b}--\cite{18b}.\\

The notion of quasi-convex functions generalizes the notion of
convex functions. More precisely, a function $f : [a, b] \to
\mathbb{R}$, is said quasi-convex on $[a, b]$ if $$f\left(
{\lambda x + \left( {1 - \lambda } \right)y} \right) \le \mathop
{\sup } \left\{ {f\left( x \right),f\left( y \right)} \right\},$$
for all $x,y \in [a,b]$ and $\lambda \in [0,1]$. Clearly, any
convex function is a quasi-convex function. Furthermore, there
exist quasi-convex functions which are not convex nor continuous.
\begin{example}
The floor function $f_{loor}(x) = \left\lfloor x \right\rfloor $,
is the largest integer not greater than $x$, is an example of a
monotonic increasing function which is quasi-convex but it is
neither convex nor continuous.
\end{example}
In the same time, one can note that the quasi-convex mappings may
be not of bounded variation, i.e., there exist quasi-convex
functions which are not of bounded variation. For example,
consider the function $f: [0,2] \to \mathbb{R}$, defined by
\begin{equation*}
f\left( x \right) = \left\{
\begin{array}{l}
 x\sin \left( {\frac{\pi }{x}} \right),\,\,\,\,x \ne 0 \\
 0,\,\,\,\,\,\,\,\,\,\,\,\,\,\,\,\,\,\,\,\,\,\,\,\,\,x = 0 \\
 \end{array} \right.,
\end{equation*}
therefore, $f$ is quasi-convex but not of bounded variation on
[0,2]. Therefore, we cannot apply the above inequalities. For
recent inequalities concerning quasi-convex mappings see
\cite{1b}--\cite{4bb}.\\

In this paper, we obtain an inequality of Simpson type via
quasi-convex mapping. This approach allows us to investigate
Simpson's quadrature rule that have restrictions on the behavior
of the integrand and thus to deal with larger classes of
functions. In general, we show that our result is better than the
classical inequality (\ref{eq1.1}).

\section{Inequalities of Simpson's type for quasi-convex functions}
In order to prove our main results, we start with the following
lemma:
\begin{lemma}
\label{lma1}Let $f^ {\prime \prime \prime}:I\subseteq
\mathbb{R}\rightarrow \mathbb{R}$ be an absolutely continuous
mapping on $I^{\circ }$, where $a,b\in I$ with $a<b$. If
$f^{(4)}\in L[a,b]$, then the following equality holds:
\begin{multline}
\int_a^b {f\left( x \right)dx}- \frac{(b-a)}{6}\left[ {f\left( a
\right) + 4f\left( {\frac{{a + b}}{2}} \right) + f\left( b
\right)} \right]
 \label{eq1} \\
=  \left( {b - a} \right)^5 \int_0^1 {p\left( t \right)f^{\left(
{4} \right)} \left( {ta + \left( {1 - t} \right)b} \right)dt},
\end{multline}
where, $$  p\left( t \right) = \left\{ \begin{array}{l}
 {\textstyle{1 \over {24}}}t^3 \left( {t - {\textstyle{2 \over 3}}} \right),\,\,\,\,\,\,\,\,\,\,\,\,\,\,\,\,\,\,\,\,\,t \in \left[ {0,{\textstyle{1 \over 2}}} \right] \\
 \\
 {\textstyle{1 \over {24}}}\left( {t - 1} \right)^3 \left( {t - {\textstyle{1 \over 3}}} \right),\,\,\,\,\,t \in \left( {{\textstyle{1 \over 2}},1} \right] \\
 \end{array} \right..
$$
\end{lemma}

\begin{proof}
We note that
\begin{align*}
I &= \left( {b - a} \right)^5 \int_0^1 {p\left( t \right)f^{\left(
4 \right)} \left( {ta + \left( {1 - t} \right)b} \right)dt}
\\
&= \frac{{\left( {b - a} \right)^5 }}{{24}}\int_0^{1/2} {t^3
\left( {t - {\textstyle{2 \over 3}}} \right)f^{\left( 4 \right)}
\left( {ta + \left( {1 - t} \right)b} \right)dt}
\\
&\qquad+ \frac{{\left( {b - a} \right)^5 }}{{24}}\int_0^{1/2}
{\left( {t - 1} \right)^3 \left( {t - {\textstyle{1 \over 3}}}
\right)f^{\left( 4 \right)} \left( {ta + \left( {1 - t} \right)b}
\right)dt}.
\end{align*}
Integrating by parts, we get
\begin{align*}
I &= \left. {{\textstyle{1 \over {24}}}t^3 \left( {t -
{\textstyle{2 \over 3}}} \right)\frac{{\,f^{\left( 3 \right)}
\left( {ta + \left( {1 - t} \right)b} \right)}}{{a - b}}}
\right|_0^{1/2} - \left. {\left( {{\textstyle{1 \over 6}}t^3  -
{\textstyle{1 \over {12}}}t^2 } \right)\frac{{\,f''\left( {ta +
\left( {1 - t} \right)b} \right)}}{{\left( {a - b} \right)^2 }}}
\right|_0^{1/2}
\\
&\qquad+ \left. {\left( {{\textstyle{1 \over 2}}t^2  -
{\textstyle{1 \over 6}}t} \right)\frac{{\,f'\left( {ta + \left( {1
- t} \right)b} \right)}}{{\left( {a - b} \right)^3 }}}
\right|_0^{1/2} - \left. {\left( {t - {\textstyle{1 \over 6}}}
\right)\frac{{\,f\left( {ta + \left( {1 - t} \right)b}
\right)}}{{\left( {a - b} \right)^4 }}} \right|_0^{1/2}
\\
& \qquad+ \int_0^{1/2} {\frac{{\,f\left( {ta + \left( {1 - t}
\right)b} \right)}}{{\left( {a - b} \right)^4 }}dt} + \left.
{{\textstyle{1 \over {24}}}\left( {t - 1} \right)^3 \left( {t -
{\textstyle{1 \over 3}}} \right)\frac{{\,f^{\left( 3 \right)}
\left( {ta + \left( {1 - t} \right)b} \right)}}{{a - b}}}
\right|_{1/2}^1
\\
&\qquad - \left. {\left( {{\textstyle{1 \over 6}}t^3  -
{\textstyle{5 \over {12}}}t^2  + \,{\textstyle{1 \over 3}}t -
{\textstyle{1 \over {12}}}} \right)\frac{{\,f''\left( {ta + \left(
{1 - t} \right)b} \right)}}{{\left( {a - b} \right)^2 }}}
\right|_{1/2}^1
\\
&\qquad+ \left. {\left( {{\textstyle{1 \over 2}}t^2  -
{\textstyle{5 \over 6}}t + \,{\textstyle{1 \over 3}}}
\right)\frac{{\,f'\left( {ta + \left( {1 - t} \right)b}
\right)}}{{\left( {a - b} \right)^3 }}} \right|_{1/2}^1 - \left.
{\left( {t - {\textstyle{5 \over 6}}} \right)\frac{{\,f\left( {ta
+ \left( {1 - t} \right)b} \right)}}{{\left( {a - b} \right)^4 }}}
\right|_{1/2}^1
\\
&\qquad+ \int_{1/2}^1 {\frac{{\,f\left( {ta + \left( {1 - t}
\right)b} \right)}}{{\left( {a - b} \right)^4 }}dt}
\end{align*}
Setting $x = ta + \left( {1 - t} \right)b$, and $dx=(a-b)dt$,
gives
\begin{eqnarray*}
\left( {b - a} \right)^5 \cdot I = \int_a^b {f\left( x \right)dx}
- \frac{(b-a)}{6}\left[ {f\left( a \right) + 4f\left( {\frac{{a +
b}}{2}} \right) + f\left( b \right)} \right],
\end{eqnarray*}
which gives the desired representation (\ref{eq1}).
\end{proof}
Therefore, we may state our main result as follows:
\begin{theorem}
\label{t2.1}Let $f^ {\prime \prime \prime}:I\subseteq
\mathbb{R}\rightarrow \mathbb{R}$ be an absolutely continuous
mapping on $I^{\circ }$ such that $f^{(4)}\in L[a,b]$, where
$a,b\in I$ with $a<b$. If $\left\vert {f^{(4)}}\right\vert $ is
quasi-convex on $[a,b]$, then the following inequality holds:
\begin{multline}
\left\vert {\int_a^b {f\left( x \right)dx}- \frac{(b-a)}{6}\left[
{f\left( a \right) + 4f\left( {\frac{{a + b}}{2}} \right) +
f\left( b \right)} \right]}\right\vert  \label{eq2.2} \\
\leq \frac{{\left( {b - a} \right)^5 }}{{5760}}\left[ {\sup
\left\{ {\left| {f^{\left( {4} \right)} \left( a \right)}
\right|,\left| {f^{\left( {4} \right)} \left( {\frac{{a + b}}{2}}
\right)} \right|} \right\} + \sup \left\{ {\left| {f^{\left( {4}
\right)} \left( {\frac{{a + b}}{2}} \right)} \right|,\left|
{f^{\left( {4} \right)} \left( b \right)} \right|} \right\}}
\right].
\end{multline}
\end{theorem}
\begin{proof}
From  Lemma \ref{lma1}, and since $f$ is quasi-convex, we have
\begin{align*}
&\left\vert {\int_a^b {f\left( x \right)dx}- \frac{(b-a)}{6}\left[
{f\left( a \right) + 4f\left( {\frac{{a + b}}{2}} \right) +
f\left( b \right)} \right]}\right\vert
\\
&= \left| {\left( {b - a} \right)^5 \int_0^1 {p\left( t
\right)f^{\left( {4} \right)} \left( {ta + \left( {1 - t}
\right)b} \right)dt} } \right|
\\
&\le \left( {b - a} \right)^5 \int_0^1 {\left| {p\left( t \right)}
\right|\left| {f^{\left( {4} \right)} \left( {ta + \left( {1 - t}
\right)b} \right)} \right|dt}
\\
&=\left( {b - a} \right)^5 \int_0^{1/2} {\left| {p\left( t
\right)} \right|\left| {f^{\left( {4} \right)} \left( {ta + \left(
{1 - t} \right)b} \right)} \right|dt}
\\
&\qquad+ \left( {b - a} \right)^5 \int_{1/2}^1 {\left| {p\left( t
\right)} \right|\left| {f^{\left( {4} \right)} \left( {ta + \left(
{1 - t} \right)b} \right)} \right|dt}
\\
&\le \left( {b - a} \right)^5 \int_0^{1/2} {\left| {{\textstyle{1
\over {24}}}t^3 \left( {t - {\textstyle{2 \over 3}}} \right)}
\right|\cdot\sup \left\{ {\left| {f^{\left( {4} \right)} \left( b
\right)} \right|,\left| {f^{\left( {4} \right)} \left( {\frac{{a +
b}}{2}} \right)} \right|} \right\}dt}
\\
&\qquad+ \left( {b - a} \right)^5 \int_{1/2}^1 {\left|
{{\textstyle{1 \over {24}}}\left( {t - 1} \right)^3 \left( {t -
{\textstyle{1 \over 3}}} \right)} \right| \cdot \sup \left\{
{\left| {f^{\left( {4} \right)} \left( {\frac{{a + b}}{2}}
\right)} \right|,\left| {f^{\left( {4} \right)} \left( a \right)}
\right|} \right\}dt}
\\
&=\frac{{\left( {b - a} \right)^5 }}{{5760}} \left[ {\sup \left\{
{\left| {f^{\left( {4} \right)} \left( a \right)} \right|,\left|
{f^{\left( {4} \right)} \left( {\frac{{a + b}}{2}} \right)}
\right|} \right\} + \sup \left\{ {\left| {f^{\left( {4} \right)}
\left( {\frac{{a + b}}{2}} \right)} \right|,\left| {f^{\left( {4}
\right)} \left( b \right)} \right|} \right\}} \right],
\end{align*}
which completes the proof.
\end{proof}

\begin{corollary}
Let $f$ as in Theorem \ref{t2.1}.
\begin{enumerate}
\item If $f$ is decreasing, then we have
\begin{multline}
\left\vert {\int_a^b {f\left( x \right)dx}- \frac{(b-a)}{6}\left[
{f\left( a \right) + 4f\left( {\frac{{a + b}}{2}} \right) +
f\left( b \right)} \right]}\right\vert  \label{eq2.4} \\
\leq \frac{{\left( {b - a} \right)^5 }}{{5760}}\left[ {\left|
{f^{\left( {4} \right)} \left( a \right)} \right| + \left|
{f^{\left( {4} \right)} \left( {\frac{{a + b}}{2}} \right)}
\right|} \right].
\end{multline}
\item If $f$ is increasing, then we have
\begin{multline}
\left\vert {\int_a^b {f\left( x \right)dx}- \frac{(b-a)}{6}\left[
{f\left( a \right) + 4f\left( {\frac{{a + b}}{2}} \right) +
f\left( b \right)} \right]}\right\vert  \label{eq2.3} \\
\leq \frac{{\left( {b - a} \right)^5 }}{{5760}}\left[ {\left|
{f^{\left( {4} \right)} \left( {\frac{{a + b}}{2}} \right)}
\right| + \left| {f^{\left( {4} \right)} \left( b \right)}
\right|} \right].
\end{multline}
\end{enumerate}
\end{corollary}

\begin{example}
Let $f:[0,1] \to \mathbb{R}$ be given by $f(x)=\exp(x^2)$. Then $
f\left( 0 \right) = 1$, $f\left( 1 \right) =
{\rm{2}}{\rm{.718281828}}$, $f\left( {\frac{1}{2}} \right) =
{\rm{1}}{\rm{.284025417}}$, $f^{\left( 4 \right)} \left( 1 \right)
= {\rm{206}}{\rm{.5894189,}}$ $f^{\left( 4 \right)} \left(
{\frac{1}{2}} \right) = {\rm{32}}{\rm{.10063542}}$, $\left\|
{f^{\left( 4 \right)} } \right\|_\infty   =
{\rm{206}}{\rm{.5894189}} $  and
 $\int_{\rm{0}}^{\rm{1}} {{\rm{exp}}\left( {{\rm{x}}^{\rm{2}} } \right)dx}  = {\rm{1}}{\rm{.462651746}}
 $.
Now, the error obtained by applying (\ref{eq1.1}) is
$0.07173243712$. However, by applying (\ref{eq2.3}) the error is
$0.03586621856$, which means that (\ref{eq2.3}) is better than
(\ref{eq1.1}).
\end{example}

\begin{corollary}
Let $f$ as in Theorem \ref{t2.1}. If $f^{(4)}$ is exits,
continuous and $ \left\| {f^{\left( 4 \right)} } \right\|_\infty
: = \mathop {\sup }\limits_{x \in \left( {a,b} \right)} \left|
{f^{\left( 4 \right)} \left( x \right)} \right| < \infty$, then
the inequality (\ref{eq2.2}) reduced to (\ref{eq1.1}).
\end{corollary}

\begin{remark}
1- We note that the constants in (\ref{eq1.1}) and (\ref{eq1.2})
are improved.
\newline

2- The corresponding version of the inequality (\ref{eq2.2}) for
powers may be done by applying the H\"{o}lder inequality and the
power mean inequality.
\end{remark}

\section{Applications to Simpson's Formula}
Let $d$ be a division of the interval $[a, b]$, i.e., $d : a = x_0
< x_1 < ... < x_{n-1} < x_n = b$, $h_i = {{\left( {x_{i+1} - x_i}
\right)} \mathord{\left/
 {\vphantom {{\left( {b - a} \right)} 2}} \right.
 \kern-\nulldelimiterspace} 2}$ and consider the Simpson's formula
\begin{align}
\label{eq4.1} S\left( {f,d} \right) = \sum\limits_{i = 0}^{n - 1}
{\frac{{f\left( {x_i } \right) + 4f\left( {x_i  + h_i } \right) +
f\left( {x_{i + 1} } \right)}}{6}\left( {x_{i + 1}  - x_i }
\right)}.
\end{align}
It is well known that if the mapping $f : [a,b] \to \textbf{R}$,
is differentiable such that $ f^{\left( 4 \right)} \left( x
\right)$ exists on $(a,b)$ and $M = \mathop {\sup }\limits_{x \in
\left( {a,b} \right)} \left| {f^{\left( 4 \right)} \left( x
\right)} \right| < \infty$, then
\begin{align}
\label{eq4.2}I = \int\limits_a^b {f\left( x \right)dx} = S\left(
{f,d} \right) + E_S\left( {f,d} \right),
\end{align}
where the approximation error $E_S\left( {f,d} \right)$ of the
integral $I$ by the Simpson's formula $ S\left( {f,d} \right)$
satisfies
\begin{align}
\label{eq4.3}\left| {E_S\left( {f,d} \right)} \right| \le
\frac{M}{{2880}}\sum\limits_{i = 0}^{n - 1} {\left( {x_{i + 1}  -
x_i } \right)^5 }.
\end{align}
In the following we give a new estimation for the remainder term
$E_S\left( {f,d} \right)$.
\begin{proposition}
\label{pro1} Let $f^{\prime \prime \prime}:I\subseteq
\mathbb{R}\rightarrow \mathbb{R}$ be an absolutely continuous
mapping on $I^{\circ }$ such that $f^{(4) }\in L[a,b]$, where
$a,b\in I$ with $a<b$. If $\left\vert {f^{(4) }}\right\vert $ is
quasi-convex on $[a,b]$, then in (\ref{eq4.2}), for every division
$d$ of $[a,b]$, the following holds:
\begin{multline*}
\left| {E_S\left( {f,d} \right)} \right| \le \frac{1}{{5760}}
\sum\limits_{i = 0}^{n - 1} {\left( {x_{i + 1}  - x_i } \right)^5
\left[ {\sup \left\{ {f^{\left( 4 \right)} \left( {x_i }
\right),f^{\left( 4 \right)} \left( {\frac{{x_i  + x_{i + 1}
}}{2}} \right)} \right\}} \right.}
\\
\left. { + \sup \left\{ {f^{\left( 4 \right)} \left( {\frac{{x_i +
x_{i + 1} }}{2}} \right),f^{\left( 4 \right)} \left( {x_{i + 1} }
\right)} \right\}} \right].
\end{multline*}
\end{proposition}
\begin{proof}
Applying Theorem \ref{t2.1} on the subintervals $[x_i,x_{i+1}]$,
$(i = 0,1, ..., n-1)$ of the division $d$, we get
\begin{multline*}
\left| {\int\limits_{x_i }^{x_{i + 1} } {f\left( x
\right)dx}-\frac{{\left( {x_{i + 1}  - x_i } \right)}}{6}\left[
{f\left( {x_i } \right) + 4f\left( {\frac{{x_i  + x_{i + 1} }}{2}}
\right) + f\left( {x_{i + 1} } \right)} \right]} \right|
\\
\le \left( {x_{i + 1}  - x_i } \right)^5 \left[ {\sup \left\{
{\left| {f^{\left( 4 \right)} \left( {x_i } \right)}
\right|,\left| {f^{\left( 4 \right)} \left( {\frac{{x_i  + x_{i +
1} }}{2}} \right)} \right|} \right\}} \right.
\\
\left. { + \sup \left\{ {\left| {f^{\left( 4 \right)} \left(
{\frac{{x_i  + x_{i + 1} }}{2}} \right)} \right|,\left| {f^{\left(
4 \right)} \left( {x_{i + 1} } \right)} \right|} \right\}} \right]
\end{multline*}
Summing over $i$ from $0$ to $n-1$ and taking into account that
$\left| {f^{(4)}} \right|$ is quasi-convex, we deduce that
\begin{multline*}
\left| {\int\limits_a^b {f\left( x \right)dx}-S\left( {f,d}
\right)} \right| \le \frac{1}{{5760}} \sum\limits_{i = 0}^{n - 1}
{\left( {x_{i + 1}  - x_i } \right)^5 \left[ {\sup \left\{ {\left|
{f^{\left( 4 \right)} \left( {x_i } \right)} \right|, \left|
{f^{\left( 4 \right)} \left( {\frac{{x_i  + x_{i + 1} }}{2}}
\right)} \right|} \right\}} \right.}
\\
\left. { + \sup \left\{ {\left| {f^{\left( 4 \right)} \left(
{\frac{{x_i + x_{i + 1} }}{2}} \right)} \right|,\left| {f^{\left(
4 \right)} \left( {x_{i + 1} } \right)} \right|} \right\}}
\right].
\end{multline*}
which completes the proof.
\end{proof}

\section{conclusion}
For fourth times continuously differentiable mapping $f$on $\left(
{a,b} \right)$ and $ \left\| {f^{\left( 4 \right)} }
\right\|_\infty  : = \mathop {\sup }\limits_{x \in \left( {a,b}
\right)} \left| {f^{\left( 4 \right)} \left( x \right)} \right| <
\infty$, the classical Simpson's inequality holds. In this paper,
we relax the conditions on Simpson's inequality; namely, the
proposed inequality (\ref{eq2.2}) holds if $f^ {\prime \prime
\prime}:I\subseteq \mathbb{R}\rightarrow \mathbb{R}$ is an
absolutely continuous mapping on $I^{\circ }$ such that $f^{\prime
\prime \prime}\in L[a,b]$ and $\left\vert {f^{(4)}}\right\vert $
is quasi-convex on $[a,b]$. In general, the inequality
(\ref{eq2.2}) is better than the classical Simpson's inequality
(\ref{eq1.1}) and the result(s) in Theorem \ref{thm1.1}, where we
give an example shows that there exist a quasi-convex mapping
which is not of bounded variation.

\centerline{}

\centerline{}


\end{document}